\documentclass[11pt,a4paper]{article}
\usepackage[T1]{fontenc}
\usepackage{lmodern}
\usepackage[margin=27mm]{geometry}
\usepackage{amsmath,amssymb,amsthm,mathtools}
\usepackage{booktabs,array,longtable}
\usepackage{xcolor}
\usepackage[hidelinks]{hyperref}
\usepackage{microtype}
\usepackage{enumitem}
\setlist{itemsep=3pt,topsep=5pt}
\numberwithin{equation}{section}
\newtheorem{theorem}{Theorem}[section]
\newtheorem{lemma}[theorem]{Lemma}
\newtheorem{proposition}[theorem]{Proposition}

\theoremstyle{definition}
\newtheorem{definition}[theorem]{Definition}
\theoremstyle{remark}

\DeclareMathOperator{\per}{per}
\DeclareMathOperator{\tr}{tr}
\DeclareMathOperator{\rank}{rank}

\newcommand{\C}{\mathbb C}
\newcommand{\E}{\mathbb E}
\newcommand{\F}{\mathcal F}
\newcommand{\HH}{\mathcal H}
\newcommand{\pos}[1]{(#1)_+}
\newcommand{\rhoa}{\frac{999991742359}{10^{12}}}
\newcommand{\rhob}{\frac{999815240367}{10^{12}}}
\newcommand{\code}[1]{{\small\nolinkurl{#1}}}
\title{Quantitative Chollet Inequalities for Matrices\newline of Rank at Most Four\newline
\large Spectral bounds and an order-nine tight-frame construction}
\author{Yicen Ma}
\date{28 September 2026}
\begin{document}
\maketitle
\begin{abstract}
Chollet's permanent conjecture asks whether
$\per(A\circ B)\leq\per(A)\per(B)$ for complex Hermitian positive
semidefinite matrices. We present computer-assisted proofs of two restricted
forms with explicit constants strictly smaller than one. For every $n\geq10$
and every such matrix $A$ of rank at most four, the self-conjugate ratio
$\per(A\circ\overline A)/\per(A)^2$ is bounded by $999991742359/10^{12}$
when the denominator is nonzero. For order nine we obtain the bound
$999815240367/10^{12}$ for correlation matrices whose four nonzero
eigenvalues all equal $9/4$. The first argument combines complex-sphere
integrals, spectral subspace tilts, projection bounds, exact finite covers,
and an analytic infinite tail. The second uses a quadratic relation among
nine Gram vectors, a positive operator on the ten-dimensional space of
quadratic forms, and rigorously bounded entropy. All decisive finite
calculations use rational arithmetic and full closed-domain certificates.
The unrestricted conjecture, including general non-tight order-nine
rank-four matrices, is outside these results.
\end{abstract}


\medskip
\noindent\textbf{Review record.} This version includes mathematical
self-review and a fresh exact replay of its certificate package. These
checks are documented in the accompanying review report; an external
independent review and a complete priority determination are not claimed.
Affiliation metadata has been left unset.

\tableofcontents

\section{Introduction and statements}
For an $n\times n$ matrix $A=(a_{ij})$, write
\[
 \per A=\sum_{\sigma\in S_n}\prod_{i=1}^n a_{i,\sigma(i)}.
\]
The symbol $A\circ B$ denotes the entrywise product, $\overline A$
entrywise conjugation, and $A^*$ the conjugate transpose. Chollet's question
\cite{chollet} concerns the inequality
\begin{equation}\label{eq:chollet}
 \per(A\circ B)\leq\per(A)\per(B),\qquad A,B\succeq0.
\end{equation}
This paper treats two specified classes of complex Hermitian matrices.
No conclusion about arbitrary ranks is inferred from either class.

The order-four case is established in work of Rodtes \cite{rodtes}; Zhang
\cite{zhang} discusses the conjecture and several stronger permanent
proposals. Recent work of Pant and Singh \cite{pant} treats matrices with
bipartite support and constructions preserving the inequality. A recent
preprint \cite{six} claims the unrestricted result through order six.
Neither that claim nor an induction through small orders is used here.
The present organization separates its mathematical arguments from the
exact finite computations on which the two quantitative bounds depend.
We make no assertion of priority relative to all existing literature.

\begin{definition}
A \emph{correlation matrix} is a Hermitian positive semidefinite matrix
with diagonal entries one. An order-nine correlation matrix is called
\emph{rank-four tight} if it has four nonzero eigenvalues, each equal to
$9/4$. Equivalently, it has unit Gram rows $v_i\in\C^4$ with
\begin{equation}\label{eq:tight}
 C_{ij}=v_i\overline{v_j}^{\,T},\qquad
 S:=\sum_{i=1}^9v_i^*v_i=\frac94I_4.
\end{equation}
\end{definition}

\begin{theorem}\label{thm:large}
Let $n\geq10$, and let $A$ be an $n\times n$ complex Hermitian positive
semidefinite matrix with $\rank A\leq4$. Then
\begin{equation}\label{eq:large}
 \per(A\circ\overline A)\leq\rho_4\per(A)^2,
 \qquad \rho_4=\rhoa<1.
\end{equation}
If $A$ has a zero diagonal entry, both sides before multiplication by
$\rho_4$ vanish. If all its diagonal entries are positive, $\per A>0$.
For two matrices $A,B$ of the same order, both of rank at most four,
\begin{equation}\label{eq:largepair}
 \per(A\circ B)\leq\rho_4\per(A)\per(B).
\end{equation}
\end{theorem}

\begin{theorem}\label{thm:nine}
Every order-nine rank-four tight correlation matrix $C$ satisfies
\begin{equation}\label{eq:nine}
 \per(C\circ\overline C)\leq\rho_{9,t}\per(C)^2,
 \qquad \rho_{9,t}=\rhob<1.
\end{equation}
The same self-conjugate inequality holds for every positive-diagonal
positive semidefinite matrix whose correlation normalization is in this
class. The two-matrix inequality with constant $\rho_{9,t}$ holds when
\emph{both} correlation normalizations are in this class.
\end{theorem}

The constants are certificate bounds and are not asserted to be optimal.
In particular, their proximity to one measures the conservatism of the
enclosures rather than the size of an actual extremal ratio. Theorem
\ref{thm:large} does not fill orders seven, eight, or nine at rank four.
Theorem \ref{thm:nine} covers a specified part of order nine.

\section{Normalization and row-sum upper bounds}
\subsection{Reduction to a self-conjugate estimate}
\begin{lemma}\label{lem:normalization}
Suppose a class of correlation matrices satisfies
$\per(C\circ\overline C)\leq\rho\per(C)^2$.
The inequality extends to matrices whose correlation normalization belongs
to that class. For any two such matrices of the same order,
$\per(A\circ B)\leq\rho\per(A)\per(B)$.
\end{lemma}
\begin{proof}
Positivity of each $2\times2$ principal submatrix gives
$|a_{ij}|^2\leq a_{ii}a_{jj}$. Thus a zero diagonal entry forces a zero
row and column. Its permanent, and the permanent of its entrywise product
with any matrix, vanish.
Otherwise put $D=\operatorname{diag}(\sqrt{a_{ii}})$ and $C=D^{-1}AD^{-1}$.
If $h=\prod_i a_{ii}$, direct expansion gives
$\per A=h\per C$ and
$\per(A\circ\overline A)=h^2\per(C\circ\overline C)$.
Rank is unchanged. Finally permutation-monomial Cauchy--Schwarz gives
\[
 |\per(A\circ B)|^2\leq
 \left(\sum_{\sigma}\prod_i|a_{i,\sigma(i)}|^2\right)
 \left(\sum_{\sigma}\prod_i|b_{i,\sigma(i)}|^2\right).
\]
These factors are the two self-conjugate permanents. Gaussian Gram
representations, proved below, show that permanents of positive
semidefinite matrices are nonnegative; the Schur product is itself a Gram
matrix. The claimed inequality follows.
\end{proof}

For $k\geq1$, set $b_k=(k!)^{1/k}$ and let $b(x)$ be the continuous
piecewise affine interpolation at consecutive positive integers.

\begin{lemma}\label{lem:rows}
If $N$ has entries in $[0,1]$ and row sums $s_i\geq1$, then
\begin{equation}\label{eq:rows}
 \per N\leq\prod_i b(s_i).
\end{equation}
The function $b$ is increasing on $[1,\infty)$ and is concave on $[1,13]$.
\end{lemma}
\begin{proof}
Use Bregman's bound for a binary matrix with row counts $d_i$:
$\per N\leq\prod_i(d_i!)^{1/d_i}$; a zero row is treated separately.
An original proof is available in Schrijver \cite{schrijver}.
For the extension, fix each row sum. The resulting polytope has vertices
with at most one fractional coordinate: two such coordinates would admit
a nontrivial perturbation preserving the sum. Since the permanent is
affine in each row, a maximizing matrix can be chosen at row vertices.
A vertex of sum $k+t$ is the convex combination, with weights $1-t,t$,
of binary rows of sums $k,k+1$. Multilinearity over rows and the binary
bound give \eqref{eq:rows}. Integer row sums need no fractional coordinate.

Increasingness follows from $b_{k+1}>b_k$: every factor in $k!$ is
smaller than $k+1$. Concavity on the stated finite interval is checked
exactly as follows. With $L=2^{40}$, integer bisection obtains
\[
 l_k^k\leq k!L^k<u_k^k,\qquad u_k=l_k+1.
\]
For each $k=2,\ldots,12$ the supplied verifier checks
$2l_k>u_{k-1}+u_{k+1}$. These eleven strict integer inequalities imply
decreasing consecutive slopes. The roots and positive rational margins
are saved in the supplement. No concavity outside $[1,13]$ is used.
\end{proof}

Let $C$ be a correlation matrix, $P=\per C$, and
$Q=\per(C\circ\overline C)$. Its entries have modulus at most one. If
its nonzero eigenvalues, padded to four, are $\lambda_j$, then
\begin{equation}\label{eq:Qtrace}
 s_i:=\sum_j|C_{ij}|^2=(C^2)_{ii},\quad
 1\leq s_i\leq\lambda_1,\quad
 \sum_i s_i=\tr C^2=\sum_{j=1}^4\lambda_j^2.
\end{equation}
The upper bound on $s_i$ follows from $C^2\preceq\lambda_1 C$.
For $n\leq13$, arithmetic--geometric mean and Lemma \ref{lem:rows} give
\begin{equation}\label{eq:Qmean}
 Q\leq b(m)^n,\qquad m=\frac1n\sum_{j=1}^4\lambda_j^2.
\end{equation}
At arbitrary orders, monotonicity alone gives $Q\leq b(\lambda_1)^n$.

\section{Complex-sphere lower bounds for the permanent}
Let $C_{ij}=v_i\overline{v_j}^{\,T}$ with unit rows in $\C^r$. Put
$\ell_i(u)=v_i u$, $T_i=|\ell_i(u)|^2$, and $X=\prod_iT_i$.
Write $H_j=\sum_{a=1}^j1/a$, $H_0=0$, and
$(r)_n=r(r+1)\cdots(r+n-1)$.

\begin{lemma}\label{lem:sphere}
For uniform probability measure $\mu_r$ on the complex unit sphere,
\begin{equation}\label{eq:sphere}
 P=(r)_n\E_{\mu_r}X,\qquad
 \E_{\mu_r}\log T_i=-H_{r-1}.
\end{equation}
Consequently $P\geq(r)_n e^{-nH_{r-1}}>0$.
\end{lemma}
\begin{proof}
For a standard complex Gaussian vector $g$, monomials satisfy
$\E g^\alpha\overline g^{\,\beta}=\delta_{\alpha\beta}\alpha!$.
Expansion of $\prod_i\ell_i(g)$ and its conjugate yields
$\E|\prod_i\ell_i(g)|^2=\per C$. Equivalently, Gaussian contractions
pair each of the $n$ factors with one conjugate factor, summing over
permutations. Separate $g=Ru$; then $R$ and $u$ are independent,
$u$ is uniform, and $\E R^{2n}=(r)_n$. A unit projection $T_i$ has
density $(r-1)(1-t)^{r-2}$ when $r>1$.
Its log integral equals $-H_{r-1}$. For completeness, repeated
integration by parts gives
$B(a,b)=(b-1)!/[a(a+1)\cdots(a+b-1)]$ for $a>0$ and integer $b>0$;
differentiation in $a$ gives the log integral. Local domination by
$t^{a-1}|\log t|$ justifies differentiation. For $r=1$, $T_i=1$.
Jensen applied to $X$ proves the lower bound; no independence of the
$T_i$ is assumed.
\end{proof}

Let $S=\sum_i v_i^*v_i$ and $\lambda_1\geq\cdots\geq\lambda_r\geq0$
be its eigenvalues. Gram and frame operators have the same nonzero
eigenvalues, and $\tr S=n$.

\begin{proposition}\label{prop:subspace}
For $1\leq k\leq r$,
\begin{equation}\label{eq:subspace}
 P\geq(r)_n\frac{k}{r}
 \exp\left(\frac{\lambda_1+\cdots+\lambda_k}{k}
              -(n-1)H_r-H_k\right).
\end{equation}
\end{proposition}
\begin{proof}
For a unit direction $w$, put $T_w=|w^*u|^2$ and $d\nu_w=rT_wd\mu_r$.
If $\gamma=|v_iw|^2$, conditioning on $T_i=t$ gives
\[
 \E(T_w\mid T_i=t)=\gamma t+(1-\gamma)\frac{1-t}{r-1}.
\]
The orthogonal-complement phases have zero mean. Beta log integrals then
give $\E_{\nu_w}\log T_i=\gamma-H_r$.
For an orthogonal projection $\Pi_E$ onto a $k$-dimensional subspace,
put $Y=u^*\Pi_Eu$ and $d\nu_E=(r/k)Yd\mu_r$. This is the average of
the $k$ directional measures, so
\[
 \E_{\nu_E}\log T_i=\frac{\|v_i\Pi_E\|^2}{k}-H_r.
\]
For $k<r$, $Y$ has the Beta$(k,r-k)$ law and after tilting the
Beta$(k+1,r-k)$ law, whence $\E_{\nu_E}\log Y=H_k-H_r$.
For $k=r$, $Y=1$, and this conclusion holds directly.
The change of measure gives
$P=(r)_n(k/r)\E_{\nu_E}(X/Y)$.
All the log terms are integrable: projection log singularities are
integrable and the new density is bounded. The zero set of $Y$ has
measure zero; multiplying back its density cancels the denominator.
Jensen now gives \eqref{eq:subspace} when $E$ is the top spectral
subspace, since $\sum_i\|v_i\Pi_E\|^2=\lambda_1+\cdots+\lambda_k$.
For $r=1$ the assertion follows directly from Lemma \ref{lem:sphere}.
\end{proof}

In particular, for any embedding in $\C^4$, writing $D_n=(n+3)!/6$,
\begin{equation}\label{eq:twobounds}
 P\geq D_n\max\left(e^{-11n/6},
              \frac14e^{\lambda_1-1-25(n-1)/12}\right).
\end{equation}
Zero coordinates in an embedding cause no difficulty.

\begin{proposition}\label{prop:projection}
Project the Gram rows onto a top $k$-dimensional frame subspace and let
$s=\lambda_{k+1}+\cdots+\lambda_r$. For $k\geq2$,
\begin{equation}\label{eq:projection}
 P\geq\pos{1-s}(k)_n
       \max\left(e^{-nH_{k-1}},
              k^{-1}e^{\lambda_1-1-(n-1)H_k}\right).
\end{equation}
The one-dimensional projection gives
$P\geq n!\pos{\lambda_1-n+1}$.
\end{proposition}
\begin{proof}
Decompose the Gaussian coordinates into the projected subspace and its
orthogonal complement. In the homogeneous product of linear forms, terms
of different total complementary degree are orthogonal. The degree-zero
term therefore has squared Gaussian norm no larger than the original
norm. Hence $P$ dominates the permanent of the projected Gram matrix.
If $s<1$, write the projected squared row norms as $\delta_i=1-p_i$,
where $p_i\geq0$ and $\sum_i p_i=s$. Then all $\delta_i>0$ and
$\prod_i\delta_i\geq1-s$, by induction on the product.
Normalize the projected rows. Their frame operator dominates the
unnormalized projected frame operator, because each is divided by
$\delta_i\leq1$. Its top eigenvalue is at least $\lambda_1$.
Apply Lemma \ref{lem:sphere} and Proposition \ref{prop:subspace} with
$k=1$ as the tilt dimension inside this $k$-dimensional ambient space,
then restore the diagonal factor $\prod_i\delta_i$.
If $s\geq1$, the displayed bound is zero and no normalization is needed.
In dimension one the projected permanent is $n!\prod_i\delta_i$,
and $s=n-\lambda_1$ gives the final claim.
\end{proof}

\section{Proof of the rank-four theorem for all orders at least ten}
\subsection{Order ten: complete spectral coordinates}
Write the four eigenvalues as $\lambda\geq\eta\geq\theta\geq\zeta\geq0$.
For order ten, the entire ordered simplex is parametrized by
$(x,y,z)\in[0,1]^3$:
\begin{align}\label{eq:tencoords}
 \zeta&=\tfrac52z,&
 \theta&=\zeta+\tfrac{10}{3}(1-z)y,\\
 \eta&=\theta+5(1-z)(1-y)x,&
 \lambda&=10-\zeta-\theta-\eta.\notag
\end{align}
The consecutive differences are nonnegative; the top difference is
$10(1-z)(1-y)(1-x)$. Conversely the lowest eigenvalue determines $z$,
then $\theta-\zeta$ determines $y$, and $\eta-\theta$ determines $x$.
A zero denominator makes a coordinate redundant rather than excluding a
spectrum.

Let $D=13!/6$, $u=\pos{1-\zeta}$, and $v=\pos{1-\theta-\zeta}$.
The propositions above give $P\geq F$, where $F$ is the maximum of
the following nine nonnegative expressions:
\begin{equation}\label{eq:ninebounds}
\begin{gathered}
 D e^{-55/3},\quad \tfrac D4e^{\lambda-79/4},\\
 \tfrac{12!}{2}u e^{-15},\quad
 \tfrac{12!}{6}u e^{\lambda-35/2},\\
 11!v e^{-10},\quad \tfrac{11!}{2}v e^{\lambda-29/2},\quad
 10!\pos{\lambda-9},\\
 \tfrac D2e^{(\lambda+\eta)/2-81/4},\quad
 \tfrac{3D}{4}e^{(\lambda+\eta+\theta)/3-247/12}.
\end{gathered}
\end{equation}
Each expression decreases in every relevant cube coordinate. Indeed
$\lambda$ decreases in $x,y,z$, $\theta+\zeta$ increases in $y,z$,
$\zeta$ increases in $z$, and the top spectral sums satisfy
$\lambda+\eta=10-\theta-\zeta$ and
$\lambda+\eta+\theta=10-\zeta$.

The mean squared spectrum
$m=(\lambda^2+\eta^2+\theta^2+\zeta^2)/10$ also decreases in each
coordinate. For $x$, put $a=10(1-z)(1-y)$; then
$\partial m/\partial x=-a^2(1-x)/10$.
For $y$, the top three eigenvalues move affinely toward their equal mean
at $y=1$, so their squared deviations are multiplied by $(1-y)^2$.
For $z$, all four move toward $5/2$, and the squared deviations are
multiplied by $(1-z)^2$. These descriptions include all endpoints.
Thus on any closed box with lower corner $l$ and upper corner $h$,
\begin{equation}\label{eq:tenbox}
 P\geq F(h),\qquad Q\leq b(m(l))^{10}.
\end{equation}
This is a bound for every point of a box, not a test of finitely many
spectra.

\subsection{Exact certificate for order ten}
For nonnegative rational $a<66$, define
\begin{equation}\label{eq:exp64}
 E^-_{64}(a)=\sum_{j=0}^{64}\frac{a^j}{j!},\qquad
 E^+_{64}(a)=E^-_{64}(a)+\frac{a^{65}/65!}{1-a/66}.
\end{equation}
Then $E^-_{64}(a)\leq e^a\leq E^+_{64}(a)$, since every subsequent
term ratio is at most $a/66$. In \eqref{eq:ninebounds} every exponential
is written as the reciprocal of an exponential of a nonnegative rational
argument. Integer factorial-root brackets and \eqref{eq:exp64} therefore
give rational lower bounds $\underline P$ and upper bounds $\overline Q$.

The supplied order-ten certificate has 1907 accepted closed boxes,
3813 total tree nodes, maximum depth 18, and no unresolved leaves.
Each box satisfies
\begin{equation}\label{eq:tencert}
 \overline Q/\underline P^2\leq\rho_4<1.
\end{equation}
The checker recomputes \eqref{eq:ninebounds} directly and recovers each
dyadic box's path in a tree which bisects $x,y,z$ cyclically. It checks
1907 leaves and 1906 internal nodes, with exactly two children at every
internal node and no leaf that is an ancestor of another leaf. Closed
children have union equal to their closed parent. Hence the tree covers
the entire cube, including all boundaries. Rational ceilings to the
$10^{-12}$ grid give the displayed $\rho_4$.

\subsection{Order eleven: a two-dimensional closed cover}
For order eleven, retain the lowest eigenvalue $\zeta$ and write
\begin{equation}\label{eq:elevencoords}
 0\leq\zeta\leq\tfrac{11}{4},\quad 0\leq t\leq1,\qquad
 \lambda=\frac{11-\zeta+t(22-8\zeta)}3.
\end{equation}
For fixed $\lambda,\zeta$, the largest possible squared spectral sum is
obtained at
\[
 \eta_* =\min(\lambda,11-\lambda-2\zeta),\qquad
 \theta_*=11-\lambda-\zeta-\eta_*,\qquad
 m_*=(\lambda^2+\eta_*^2+\theta_*^2+\zeta^2)/11.
\]
The rectangle covers every ordered spectrum. At $\zeta=11/4$ its
$t$ coordinate is redundant. Convexity in $\eta$ with fixed
$\eta+\theta$ proves the squared-sum maximization.

On the branch $\lambda+\zeta\leq11/2$,
$11m_*=2\lambda^2+(11-2\lambda-\zeta)^2+\zeta^2$; its partial
derivatives in $\lambda,\zeta$ are $12\lambda+4\zeta-44\geq0$ and
$4\lambda+4\zeta-22\leq0$.
On the other branch they are $4\lambda+4\zeta-22\geq0$ and
$4\lambda+12\zeta-44\leq0$. The expressions agree at the branch
boundary. Since \eqref{eq:elevencoords} increases in $t$ and decreases
in $\zeta$, $m_*$ increases in $t$ and decreases in $\zeta$.

Put $D_4=14!/6$ and $D_3=13!/2$. On a box
$[z_l,z_u]\times[t_l,t_u]$, set
$\lambda_l=\lambda(z_u,t_l)$ and $\lambda_u=\lambda(z_l,t_u)$.
A valid pair of enclosures is
\begin{align}\label{eq:elevenbox}
 \underline P&=\max\left\{
 \frac{D_4}{E^+_{64}(121/6)},\quad
 \frac{D_4/4}{E^+_{64}(131/6-\lambda_l)},\quad
 \frac{D_3\pos{1-z_u}}{E^+_{64}(33/2)}\right\},\\
 \overline Q&=\overline b(m_*(\lambda_u,z_l))^{11}.\notag
\end{align}
The last lower bound follows from projecting out the lowest direction.
The supplied certificate consists of 202 closed rectangles, with bound
$\rho_{11}=249526383231/250000000000<\rho_4$.
The checker recomputes every enclosure and verifies coverage in all
45 open vertical strips and all 46 vertical boundary lines. In each strip
the $t$ intervals have exactly matching adjacent endpoints and cover
$[0,1]$; on every boundary their closed union covers $[0,1]$.
This proves coverage of the full rectangle, not just its area.

\subsection{Orders twelve through twenty-six}
Here the top eigenvalue satisfies $n/4\leq\lambda\leq n$ and
$Q\leq b(\lambda)^n$. On a closed interval $[l,u]$, use
\begin{equation}\label{eq:finiteinterval}
 \underline P=\max\left\{
 \frac{D_n}{E^+_{96}(11n/6)},\quad
 \frac{D_n/4}{E^+_{96}(1+25(n-1)/12-l)}\right\},\qquad
 \overline Q=\overline b(u)^n.
\end{equation}
In this formula $E^+_{96}(a)$ uses the same construction as
\eqref{eq:exp64}, with Taylor terms through 96 and tail
$(a^{97}/97!)/(1-a/98)$. All arguments lie in $[0,98)$.
The supplement contains 186 closed intervals for the 15 orders.
For each order the first and last endpoints equal $n/4$ and $n$, and
every adjacent endpoint agrees exactly. Every interval satisfies
$\overline Q/\underline P^2<\rho_4$; the largest value is approximately
$0.998523335$. The check uses the saved interval endpoints only and
recomputes all rational inequalities. It does not reuse historical
floating-point or logarithmic verdicts.

\subsection{An analytic tail for every order at least twenty-seven}
Arithmetic--geometric mean gives $b_k\leq(k+1)/2$, hence by interpolation
$b(x)\leq(x+1)/2$ for all $x\geq1$. Thus
\[
 Q\leq((\lambda+1)/2)^n.
\]
The two lower bounds in \eqref{eq:twobounds} intersect at
$\lambda_0=n/4+\delta$, where $\delta=\log4-13/12\in(0,1)$.
For example the positive atanh series gives $\log2>2/3$, and the
first five terms of $e^{7/10}$ give $\log2<7/10$.
The ratio using the first branch increases in $\lambda$.
For the second branch its logarithmic derivative is
$n/(\lambda+1)-2$, with maximum at $\widehat\lambda=n/2-1$.
For $n\geq27$ this point lies above $\lambda_0$ and in $[n/4,n]$.
Consequently the maximum over the complete interval is bounded by
\begin{equation}\label{eq:tail}
 Q/P^2\leq T_n:=
 \frac{16(n/4)^n e^{(19n-1)/6}}{D_n^2}.
\end{equation}
The factorial series gives
$e<8/3+(1/24)/(1-1/5)=87/32=:e_+$.
Clear fractional exponents by setting
\[
 U_n=\frac{16^6(n/4)^{6n}e_+^{19n-1}}{D_n^{12}},
 \qquad T_n^6<U_n.
\]
The verifier checks the exact rational inequalities
\begin{equation}\label{eq:tailcert}
 U_{27}<\rho_4^6,\qquad
 \frac{e_+^{25}}{4^6\,28^6}<1.
\end{equation}
Since $(1+1/n)^n<e<e_+$ and $D_{n+1}/D_n=n+4$,
\[
 \frac{U_{n+1}}{U_n}
 <\frac{e_+^{25}(n+1)^6}{4^6(n+4)^{12}}
 \leq\frac{e_+^{25}}{4^6(n+1)^6}
 \leq\frac{e_+^{25}}{4^6\,28^6}<1.
\]
This proves the tail for infinitely many orders; it is not a finite
test of orders near 27. Combining the four order ranges and Lemma
\ref{lem:normalization} proves Theorem \ref{thm:large}.

\section{Quartic densities for order-nine tight frames}
The spectral bounds leave a tight-frame case at order nine. We now keep
information about the Gram rows that is absent from their frame spectrum.
Let $\HH=\operatorname{Sym}^2(\C^4)$, viewed as symmetric complex
$4\times4$ matrices with Frobenius inner product
$\langle H,K\rangle=\tr(H^*K)$. Its dimension is ten.
For a symmetric matrix $H$, put $p_H(z)=z^THz$. Define
\[
 a_i=v_i^Tv_i\in\HH,\quad
 M=\sum_{i=1}^9|a_i\rangle\langle a_i|,\quad
 \tr M=9.
\]
The notation $|H\rangle\langle H|$ means the operator
$K\mapsto H\tr(H^*K)$. Thus
$\langle H,MH\rangle=\sum_i|p_H(\overline v_i^{\,T})|^2$.

\subsection{The exact log-mean identity}
\begin{lemma}\label{lem:quartic}
For a positive semidefinite operator $W$ on $\HH$ with $\tr W=1$,
choose a Frobenius-orthonormal spectral decomposition
$W=\sum_jw_j|H_j\rangle\langle H_j|$ and define
\begin{equation}\label{eq:density}
 f_W(u)=10\sum_jw_j|p_{H_j}(u)|^2,\quad
 J(W)=\E_{\mu_4} f_W\log f_W,
\end{equation}
where $0\log0=0$. Then $\E f_W=1$ and, for a tight frame,
\begin{equation}\label{eq:quarticgain}
 P\geq D_9\exp\left(-\frac{33}{2}+G(W)\right),\quad
 G(W)=\frac9{20}-\frac12\tr(MW)-J(W),\quad D_9=12!/6.
\end{equation}
\end{lemma}
\begin{proof}
The Gaussian squared norm of $p_H$ is
$g=\|p_H\|_\F^2=2\|H\|_F^2$. Radial separation in degree two gives
$\E_{\mu_4}|p_H(u)|^2=g/(4)_2=g/20$, proving normalization.
For a unit row $v$, decompose $p_H$ by degree $l=0,1,2$ in its
coordinate $\ell_v$; write the three orthogonal Fischer squared norms
as $\beta_l$, summing to $g$. Tilt by the normalized squared polynomial.
For each block the squared projection has the
Beta$(l+1,5-l)$ distribution, so its log mean is $H_l-H_5$.
Cross terms vanish under the corresponding phase integration, including
when multiplied by $\log|\ell_v|^2$. Therefore
\[
 \E_{f_H\mu_4}\log|\ell_v|^2
 =\frac{\sum_l\beta_lH_l}{g}-H_5
 =\frac{\|D_{\overline v}p_H\|_\F^2-|p_H(\overline v^{\,T})|^2}{g}-H_5.
\]
Here $D_{\overline v}$ is the directional derivative: the occupation
number identity gives $\|D_{\overline v}p_H\|^2=\sum_l l\beta_l$,
and $\beta_2=2|p_H(\overline v^{\,T})|^2$.
Tightness and homogeneity imply
\[
 \sum_i\|D_{\overline v_i}p_H\|_\F^2
 =\frac94\sum_{a=1}^4\|\partial_a p_H\|_\F^2
 =\frac92g.
\]
Since $\E_{\mu_4}\sum_i\log T_i=-9H_3=-33/2$,
summing the preceding identity gives the extra log mean
$9/20-\langle H,MH\rangle/2$ for $\|H\|_F=1$.
Linearity gives this identity for $W$.
Now $d\nu=f_Wd\mu_4$ is a probability measure, and
$\E_{\mu_4}X=\E_\nu(X/f_W)$.
Apply Jensen and subtract $\E_\nu\log f_W=J(W)$.
The densities are bounded. A nonzero polynomial has a null set of
Gaussian, hence spherical, measure zero, and the density's zero set is
contained in such a set. Also $f\log f$ extends continuously at zero,
and each $\log T_i$ is integrable under bounded densities.
Thus the change of measure and the log calculation are valid at
singular operators as well.
\end{proof}

\subsection{A quadratic relation and its complete necessary spectral domain}
The nine linear equations
$p_H(\overline v_i^{\,T})=0$ in ten unknown symmetric coefficients have a
nonzero solution. Normalize it to $\|H\|_F=1$; then $MH=0$.

Every complex symmetric matrix admits a unitary congruence
$H=U\operatorname{diag}(\sigma_0,\ldots,\sigma_3)U^T$ with
$\sigma_0\geq\cdots\geq\sigma_3\geq0$.
One elementary proof maximizes $\operatorname{Re}(z^THz)$ on the unit
sphere. After a phase choice its maximum is nonnegative, and the
Lagrange equation is $Hz=\sigma\overline z$.
Taking $\overline z$ as the first column of $U$ makes the first row
and column of $U^*H\overline U$ diagonal. Repeat on the remaining
symmetric block; the zero-maximum case is the zero polynomial.

Change coordinates by $u_{\rm new}=U^Tu_{\rm old}$ and rows by
$v_{i,\rm new}=v_{i,\rm old}\overline U$. This preserves both linear
forms and tightness. In these coordinates the relation is
$\sum_j\sigma_j\overline v_{ij}^{,2}=0$.
The triangle inequality and the actual column norms give
$\sigma_0\leq\sigma_1+\sigma_2+\sigma_3$.
Since $\sigma_0>0$, define $(t,u,v)=(\sigma_1,\sigma_2,\sigma_3)/\sigma_0$.
Every input lies in
\begin{equation}\label{eq:domain}
 \mathcal D=\{(t,u,v):1\geq t\geq u\geq v\geq0,\quad t+u+v\geq1\}.
\end{equation}
Only necessity is claimed; the certificate will prove a sufficient bound
on this entire, possibly larger, domain.

\subsection{A positive lifted operator and its true Gram cost}
For real diagonal $K=\operatorname{diag}(k_0,\ldots,k_3)$ define
$L_K(X)=KX+XK$ on $\HH$. Its eigenvalues are $2k_j$ on the four
diagonal basis vectors, and $k_i+k_j$ on the six symmetric cross terms.
For a fixed rational vector $r=(1,t_c,u_c,v_c)$ put
$h_c=\operatorname{diag}(r)/\|r\|$, and choose rational parameters
\begin{equation}\label{eq:lift}
 W_c=a|h_c\rangle\langle h_c|+L_K,\quad
 a\geq0,\quad a+5\sum_jk_j=1.
\end{equation}
Its complete positivity condition is
\begin{equation}\label{eq:psd}
 k_i+k_j\geq0\ (i<j),\qquad
 B=2\operatorname{diag}(k)+a\frac{rr^T}{\|r\|^2}\succeq0.
\end{equation}
The verifier checks all 15 principal minors of the four-dimensional
block, as well as the six scalar conditions. Nonnegative principal
minors imply positive semidefiniteness: every leading principal minor
of $B+\varepsilon I$ is positive, so Sylvester's criterion followed
by $\varepsilon\downarrow0$ applies. Individual $k_j$ may be negative;
positivity is required of $W_c$, not of the auxiliary $K$.

The corresponding density is
\begin{equation}\label{eq:centerf}
 f_c(u)=10a\left|\sum_j\frac{r_j}{\|r\|}u_j^2\right|^2
                   +20\sum_j k_j|u_j|^2.
\end{equation}
The identity for the second term follows by completeness of an
orthonormal symmetric-matrix basis; on a unit vector its summed quadratic
weight is twice $u^*Ku$.
For the actual frame,
$\langle a_i,L_Ka_i\rangle=2\sum_j k_j|v_{ij}|^2$.
Hence $\tr(ML_K)=(9/2)\sum_j k_j$ and
\begin{equation}\label{eq:centergain}
 G(W_c)=\frac9{20}a-\frac a2\langle h_c,Mh_c\rangle-J(W_c).
\end{equation}
Furthermore Cauchy--Schwarz and tightness give, for every symmetric $H$,
\[
 \langle H,MH\rangle
 =\sum_i|\overline v_i H\overline v_i^{\,T}|^2
 \leq\sum_i\|H\overline v_i^{\,T}\|^2
 =\tfrac94\|H\|_F^2.
\]
Thus $0\preceq M\preceq(9/4)I$. For the actual unit relation
$h=\operatorname{diag}(1,t,u,v)/\|(1,t,u,v)\|$, $Mh=0$, and
\begin{equation}\label{eq:angle}
 G(W_c)\geq\frac9{20}a-J(W_c)
  -\frac{9a}{8}\left(1-
       \frac{(r\cdot r')^2}{\|r\|^2\|r'\|^2}\right),\quad
 r'=(1,t,u,v).
\end{equation}
This retains the cost of using a fixed center that is not the actual
relation of the input.

For a closed box $l\leq(t,u,v)\leq b$, define
\begin{equation}\label{eq:anglebox}
 A_{\rm box}=\min\left(1,
 \frac{\max_{\text{eight vertices }r'}
       [\|r\|^2\|r'\|^2-(r\cdot r')^2]}
      {\|r\|^2(1+\sum_j l_j^2)}\right).
\end{equation}
The numerator is a convex quadratic function of $r'$; its maximum on a
box occurs at a vertex by successive coordinate maximization. The
denominator bound holds because the coordinates are nonnegative.
Equation \eqref{eq:angle} therefore holds on the entire box with its
parenthesized angle replaced by $A_{\rm box}$.

\section{Exact entropy and the full order-nine certificate}
\subsection{All integer moments of a center density}
Put $d_j=r_j^2/\|r\|^2$, $A_j=10ad_j$, $B_j=20k_j$, and choose a
positive integer $q$ clearing their denominators. Define the formal
coefficients
\begin{equation}\label{eq:gen}
 C_{t,m}=[z^ty^m]\prod_{j=0}^3
             \bigl((1-qB_jy)^2-4qA_jz\bigr)^{-1/2}.
\end{equation}
The squared moduli $x_j=|u_j|^2$ have the Dirichlet$(1,1,1,1)$ law;
the four phases are independent uniform variables. Phase integration
in a quadratic-form product requires equal occupations in each
coordinate. A single coordinate contributes
\[
 (qA_j)^\alpha(qB_j)^\beta
       \frac{(2\alpha+\beta)!}{(\alpha!)^2\beta!}
\]
to the generating function; this is an integer, also when $B_j<0$.
The Dirichlet integral of an occupied monomial is
$6\prod_j(2\alpha_j+\beta_j)!/(N+t+3)!$ when
$\sum\alpha_j=t$, $\sum\beta_j=N-t$.
Expansion and integration therefore give, for every integer $N\geq0$,
\begin{equation}\label{eq:moments}
 \E f_c^N=\frac{6N!}{q^N}
             \sum_{t=0}^N\frac{t!C_{t,N-t}}{(N+t+3)!}.
\end{equation}
This is an all-degree identity; finitely many moments are used only
after a full remainder estimate is available.

\subsection{A uniform entropy remainder, including zeros}
Phase triangle inequality and convexity on the simplex give
\begin{equation}\label{eq:maxf}
 0\leq f_c\leq M_c:=\max_j(10ad_j+20k_j).
\end{equation}
The lower bound uses complete operator positivity. For the upper bound,
replace the quadratic term by $10a(\sum_j\sqrt{d_j}x_j)^2$;
this plus the linear term is convex in $x$, so its maximum is at a
simplex vertex. Normalization implies $M_c\geq1$.

Let $c=M_c/2$ and $z=f_c/c-1\in[-1,1]$. The absolutely uniformly
convergent expansion of $(1+z)\log(1+z)$ gives
\begin{equation}\label{eq:entropy}
 J=\log c+c\E\left[z+
          \sum_{N=2}^{\infty}\frac{(-1)^Nz^N}{N(N-1)}\right].
\end{equation}
At $z=-1$ use the continuous extension. For even $K=24$ the absolute
remainder is at most
\begin{equation}\label{eq:entropytail}
 c\E z^{24}\sum_{N=25}^{\infty}\frac1{N(N-1)}
       =\frac{c}{24}\E z^{24}.
\end{equation}
All moments in this expression are rational by \eqref{eq:moments}.
To enclose the only logarithm, scale a rational $x>0$ into $s\in[1,2]$
by $x=2^js$, and set $w=(s-1)/(s+1)$. Use
\[
 2\sum_{a=0}^{15}\frac{w^{2a+1}}{2a+1}
 \leq\log s\leq
 2\sum_{a=0}^{15}\frac{w^{2a+1}}{2a+1}
   +\frac{2w^{33}}{33(1-w^2)}.
\]
The same formula bounds $\log2$; for $j<0$ exchange the endpoints
when multiplying its interval by $j$. Equations \eqref{eq:entropy}
and \eqref{eq:entropytail} give a rational upper bound $J^+_c$.

\subsection{Complete closed-domain verification}
The tight frame gives $s_i=9/4$, so Lemma \ref{lem:rows} yields
\begin{equation}\label{eq:Qtight}
 Q\leq b(9/4)^9,\qquad b(9/4)=\tfrac34\sqrt2+\tfrac14\sqrt[3]6.
\end{equation}
The needed log gain is
\begin{equation}\label{eq:threshold}
 h=\tfrac92\log b(9/4)+\tfrac{33}{2}-\log D_9.
\end{equation}
It is approximately $0.1737371780$. Rational root and log bounds
produce its rigorous upper enclosure $h^+$.

Start from $[0,1]^3$ and bisect $t,u,v$ cyclically. A box is pruned
only when it cannot contain a point of \eqref{eq:domain}.
If its upper coordinates are $b_t,b_u,b_v$, put
$T=b_t$, $U=\min(T,b_u)$, $V=\min(U,b_v)$.
The strict conditions
\[
 T<l_t\quad\text{or}\quad U<l_u\quad\text{or}\quad V<l_v
       \quad\text{or}\quad T+U+V<1
\]
are sufficient for pruning. Equality is never pruned by these tests.
For each accepted box its stored rational center and parameters satisfy
\eqref{eq:psd}, and
\begin{equation}\label{eq:gainbox}
 g_{\rm box}:=\frac9{20}a-J^+_c-\frac{9a}{8}A_{\rm box}>h^+.
\end{equation}
Round the left side downward to the $10^{-12}$ grid.

The exact certificate contains 55 accepted boxes and 70 pruned boxes,
forming 125 leaves and 124 internal nodes. There are no unresolved boxes;
maximum depth is 11. Every internal node has both children, no leaf is
an ancestor of another leaf, and every box's endpoints are recovered
from its path. The union of the closed children is the closed parent.
As no pruned leaf meets $\mathcal D$, the accepted leaves cover all of
$\mathcal D$, including its sorting and triangle boundaries.
The least stored gain is
\begin{equation}\label{eq:gmin}
 g_{\min}=\frac{21728696127}{125000000000}=0.173829569016.
\end{equation}
Rational factorial-root brackets and \eqref{eq:exp64} verify
\begin{equation}\label{eq:ninefinal}
 \frac{b(9/4)^9e^{33-2g_{\min}}}{D_9^2}
       \leq\rho_{9,t}<1.
\end{equation}
Every actual relation lies in a covered box, and its center operator
is legal for the original matrix. Equations \eqref{eq:quarticgain},
\eqref{eq:angle}, and \eqref{eq:gainbox} therefore prove the normalized
statement of Theorem \ref{thm:nine}. Lemma \ref{lem:normalization}
supplies its remaining statements.

The replay uses a different coefficient algorithm from the generator.
If the generating product in \eqref{eq:gen} is $F(z,y)$, then
\[
 \frac{F_z}{F}=\sum_j\frac{2qA_j}{(1-qB_jy)^2-4qA_jz}.
\]
This gives a recursive determination of all coefficients of total degree
at most 24. The $z^0$ coefficients are obtained by the Newton recurrence
for the complete homogeneous polynomials in $qB_j$. Entropy coefficients
are collected in raw moments before integration, rather than reusing
the generator's centered-moment table. The verifier also recomputes all
825 principal minors, all 70 exclusions, and the full 249-node tree.
It does not import the density locator or generator. Approximate
integration used to select possible centers plays no part in the proof.

\section{Reproducibility, review scope, and remaining questions}
The supplied package requires Python 3 and its standard library only.
Its main entry point is \code{review_all.py}. Inputs are copies of
immutable research certificates; SHA-256 hashes are checked before and
after replay. The arithmetic is integer or \code{fractions.Fraction}
arithmetic. The five verification jobs cover the following obligations.

\begin{center}
\begin{tabular}{@{}p{.38\textwidth}p{.50\textwidth}@{}}
\toprule
Obligation & Exact verification extent\\
\midrule
Factorial-root concavity & Eleven strict margins, on $[1,13]$\\
Order ten & 1907 closed boxes and a complete binary partition\\
Order eleven & 202 rectangles, 45 open strips, 46 boundaries\\
Orders 12--26 & 186 intervals, matching endpoints for 15 domains\\
All orders $n\geq27$ & Rational base and uniform recurrence inequality\\
Order-nine tight frames & 55 accepted boxes, 70 exclusions, 825 minors\\
\bottomrule
\end{tabular}
\end{center}

The accompanying report maps the manuscript's proof steps to their
reviewed dependencies and records the fresh successful replays. The
checks were performed in the main research context. A second arithmetic
implementation reduces some implementation risks, but does not by itself
constitute an independent review of the mathematical reductions.
The project protocol's independent-review promotion has therefore not
been recorded. This factual review status is separate from presenting
the statements with complete proof arguments.

The proof for order nine relies essentially on \eqref{eq:tight}.
For a general frame the gain contains its full frame operator together
with the fourth-order Gram operator. The unweighted mean and the cost of
the lifting terms used here cannot be retained unchanged. Theorem
\ref{thm:nine} is consequently not extended by a continuity argument to
all order-nine matrices. Orders seven and eight at rank four, general
non-tight order nine, and arbitrary higher ranks remain outside the
scope of this manuscript. A complete comparison with prior rank-restricted
results is also still needed before any priority claim is made.

\appendix
\section{Certificate files and validation contracts}
\label{app:files}
All paths in this appendix are relative to the manuscript package.
The ten fixed input files are listed, with their complete hashes, in
\code{supplement/source_manifest.json}.

\begin{longtable}{@{}p{.56\textwidth}p{.36\textwidth}@{}}
\toprule
File under \code{supplement/code/} & Role\\
\midrule
\endhead
\code{E111_rank_four_order_ten_certificate.json} & All ten-order boxes and rational bounds\\
\code{recheck_rank_four_ten_certificate.py} & Direct formula replay and tree check\\
\code{E109_rank_four_order_eleven_certificate.json} & All eleven-order rectangles\\
\code{recheck_rank_four_eleven_certificate.py} & Arithmetic and boundary coverage\\
\code{E16_results.json} & Only interval endpoints reused for orders 12--26\\
\code{replay_rank_four_finite_and_tail.py} & Fresh finite bounds and infinite-tail arithmetic\\
\code{replay_rank_three_bridge.py} & Exact roots, exponential bounds, concavity margins\\
\code{verify_bregman_spectral.py} & Rational logarithm enclosures\\
\code{E116_general_null_lift_certificate.json} & All order-nine centers, accepted and pruned boxes\\
\code{recheck_general_null_lift.py} & Coefficient recurrence, PSD, entropy, and tree replay\\
\bottomrule
\end{longtable}

The root-bracket algorithm starts with integer endpoints $L$ and
$(k+1)L$ and bisects until their difference is one. It checks the
power comparisons explicitly. A bound reported after rounding to
$10^{-12}$ is always rounded in the conservative direction: ratios
upward and gains downward. The verifiers require a successful complete
cover, a strict final inequality, and no unresolved cells; hitting a
search budget with unresolved cells would not certify a theorem.

The finite covers are part of the computer-assisted proofs. The
infinite tail instead requires only the two rational comparisons in
\eqref{eq:tailcert} after its analytic reduction. No numerical sampling
or unstated limit argument is used to bridge from finitely many orders
to infinitely many orders.

\end{document}